\documentclass[reqno]{amsart}
\usepackage{amssymb,amsmath,amsfonts,amsthm,amsaddr} 
\usepackage{enumitem}    
\usepackage{verbatim}
\usepackage{hyperref}    
\usepackage{graphicx}  
\hypersetup{colorlinks=true,linkcolor=green,pdfborderstyle={/S/U/W 1}}

\newcommand{\floor}[1]{{\lfloor #1 \rfloor}}
\newcommand{\ceil}[1]{{\lceil #1 \rceil}}

\theoremstyle{plain}
\newtheorem{theorem}  {Theorem}  [section]
\newtheorem{lemma}  [theorem]   {Lemma}
\newtheorem{corollary}[theorem] {Corollary}
\newtheorem{fact} [theorem] {Fact}
\newtheorem{proposition} [theorem] {Proposition}
\newtheorem{claim} [theorem] {Claim}

\theoremstyle{definition}

\newtheorem{definition}[theorem] {Definition}
\newtheorem{example}[theorem] {Example}

\newcommand{\claimproof}{\renewcommand{\qedsymbol}{$\diamond$}}

\usepackage[numbers,sort&compress]{natbib}

\DeclareMathOperator{\Ext}{Ext}
\DeclareMathOperator{\Hom}{Hom}
\DeclareMathOperator{\bCol}{{Col}}
\DeclareMathOperator{\Col}{{Col}}
\DeclareMathOperator{\bHom}{{Hom}}
\DeclareMathOperator{\Mix}{Mix}

\DeclareMathOperator{\SP}{SP}
\DeclareMathOperator{\SPR}{SPR}
\DeclareMathOperator{\CSP}{CSP}
\DeclareMathOperator{\PSC}{PSC}
\DeclareMathOperator{\NU}{NU}

\DeclareMathOperator{\Poly}{P}
\DeclareMathOperator{\PSPACE}{PSPACE}
\DeclareMathOperator{\Recon}{Recon}
\newcommand{\kNU}[1][k]{#1-$\NU$}
\DeclareMathOperator{\id}{id}
\DeclareMathOperator{\odd}{odd}
\newcommand{\ME}{\Mix_{\Ext}} 
\newcommand{\MH}{\Mix_{\Hom}} 
\DeclareMathOperator{\RE}{\Recon_{\Ext}}
\DeclareMathOperator{\RH}{\Recon_{\Hom}}
\newcommand{\RSP}{\SPR} 

 \usepackage{tikz} 
 \usepackage{twoopt}
 \tikzset{
  vert/.style={circle, draw=black!100,fill=black!100,thick, inner sep=0pt, minimum size=2mm}, 
  empty/.style={draw=none, fill=none, minimum size=0mm, inner sep=0pt},
 }
 \newcommand{\Vlabel}[4][.3cm]{\draw node[empty, #3 of = #2, node distance = #1] () {$#4$};}
 
\begin{document}
\thispagestyle{empty}  
\title[A Brightwell-Winkler type characterisation of NU graphs]{A Brightwell-Winkler type characterisation of ${\rm NU}$ graphs}

\author{Mark Siggers}
\address{Kyungpook National University, Daegu, Republic of Korea.}
\email{mhsiggers@knu.ac.kr}
\thanks{The author was supported by Korean NRF Basic Science Research Program (NRF-2022R1A2C1091566) funded by the Korean government (MEST)
          and the Kyungpook National University Research Fund.}


\begin{abstract}
In 2000,  Brightwell and Winkler characterised dismantlable graphs as the graphs $H$ for which the Hom-graph ${\rm Hom}(G,H)$, defined on the set of homomorphisms from $G$ to $H$, is connected for all graphs $G$.  This shows that the reconfiguration version ${\rm Recon_{Hom}}(H)$ of the $H$-colouring problem, in which one must decide for a given $G$ whether ${\rm Hom}(G,H)$ is connected, is trivial if and only if $H$ is dismantlable.      
  
 We prove a similar starting point for the reconfiguration version of the $H$-extension problem.   Where ${\rm Hom}(G,H;p)$ is the subgraph of the Hom-graph ${\rm Hom}(G,H)$ induced by the $H$-colourings extending the $H$-precolouring $p$ of $G$, the reconfiguration version ${\rm Recon_{Ext}(H)}$ of the $H$-extension problem asks, for a given  $H$-precolouring $p$ of a graph $G$, if ${\rm Hom}(G,H;p)$ is connected.   We show that the graphs $H$ for which ${\rm Hom}(G,H;p)$ is connected for every choice of $(G,p)$ are exactly the ${\rm NU}$ graphs.  This gives a new characterisation of ${\rm NU}$ graphs, a nice class of graphs that is important in the algebraic approach to the ${\rm CSP}$-dichotomy.  
  
We further  give bounds on the diameter of ${\rm Hom}(G,H;p)$ for ${\rm NU}$ graphs $H$, and show that shortest path between two vertices of ${\rm Hom}(G,H;p)$ can be found in parameterised polynomial time. We apply our results to the problem of shortest path reconfiguration, significantly extending recent results. 
  
 \end{abstract}

\maketitle%


\section{Introduction}

  In this paper,  graphs may have loops but no multiple edges.  Graphs without loops are {\em irreflexive} and graphs in which every vertex has a loop are {\em reflexive}. A path in a graph $H$ between vertices $u$ and $v$ is called a $uv$-path. 

  Recall that a homomorphism $\phi: G \to H$ from a graph $G$ to a graph $H$, is a vertex map $\phi: V(G) \to V(H)$ that preserves edges:
      \[ u \sim v \Rightarrow \phi(u) \sim \phi(v). \]
The $\Hom$-graph $\Hom(G,H)$ is the graph on the set of $H$-colourings of $G$ in which two such colourings $\phi$ and $\psi$ are adjacent if  they satisfy the following for all pairs of vertices $u$ and $v$  of $G$:
  \[ u \sim v \Rightarrow \phi(u) \sim \phi'(v). \]

While the $H$-colouring problem $\Hom(H)$ asks for an instance graph $G$ if the graph $\Hom(G,H)$ is non-empty, the mixing version $\MH(H)$, called $H$-mixing, asks for an instance $G$ if $\Hom(G,H)$ is connected, and the  reconfiguration version $\RH(H)$, known as $H$-recolouring, asks for an instance $(G,\phi,\psi)$ if there is a path from $\phi$ to $\psi$ in $\Hom(G,H)$.  

 We say that a problem is tractable, or in $\Poly$, if it can be solved in polynomial time, and say that is in $\PSC$ if it is $\PSPACE$-complete.

There have been several papers in the last 10-15 years determining the computational complexity of $\RH(H)$ and $\MH(H)$ for various graphs $H$. Many of these papers support the assertion that the complexity of the problems is closely related to topological properties of the graph $H$, or rather, of various simplicial complexes derived from $H$.  For example, it is shown in \cite{Wroch20} that $\RH(H)$ is tractable for an irreflexive graph $H$ if its neighbourhood complex is purely $1$-dimensional. It is shown in \cite{LNS20} that $\RH(H)$ is in $\PSC$ for reflexive $H$ if it clique complex is a $2$-sphere.  It is shown in \cite{KLS} that $\MH(H)$ is coNP-complete for reflexive $H$ if its clique complex purely $1$-dimension and has non-trivial first homology.

All of these results are preceded by a result of Brightwell and Winkler from \cite{BW00}. We define a graph $H$ to be {\em $\Hom$-trivial} if $\Hom(G,H)$ is connected for all $G$. We use the word `trivial' as in this case all instances of $\RH(H)$ and $\MH(H)$ are YES-instances, so both problems are trivial.    
 Dismantlable graphs are topologically trivial in that their clique complex deformation retracts to a point.  In \cite{BW00}, a paper well known for its characterisation of cop-win graphs as dismantlable graphs, Brightwell and Winkler gave several characterisations of dismantlable graphs. The following is the characterisation we are most interested in. 

\begin{theorem}[\cite{BW00}]\label{thm:BW00}
  A graph is dismantlable if and only if it is $\Hom$-trivial. 
\end{theorem}
  
 An {\em $H$-precolouring $p$ of $G$} is a map $p: S \to V(H)$ for some subset $S \subset V(G)$ of its vertices.  A homomorphism in $\bHom(G,H)$ {\em extends} $p$ if it restricts to $p$ on $S$. We usually make the support $S$ of $p$ implicit, and write $|p|$ for $|S|$.  

\begin{definition}
    For an $H$-precolouring $p$ of $G$, let $\bHom(G,H;p)$ be the subgraph of $\bHom(G,H)$ induced by the vertices that extend $p$. 
 \end{definition}
 
   The well known homomorphism extension problem $\Ext(H)$ ask for an instance $(G,p)$, where $p$ is an $H$-precolouring of $G$, if $\bHom(G,H;p)$ has a vertex. We make the following definitions. 

\begin{definition}
 The {\em $H$-extension reconfiguration problem} $\RE(H)$ asks for an instance $(G,p,\phi,\psi)$ with $\phi$ and $\psi$ in $\bHom(G,H;p)$ if $\phi$ and $\psi$ are in the same component of $\bHom(G,H;p)$. 
A graph $H$ is {\em $\Ext$-trivial} if $\bHom(G,H;p)$ is connected for every pair $(G,p)$ where $p$ is an $H$-precolouring of a graph $G$.
\end{definition}

Near-unanimity graphs, or $\NU$ graphs, are an important class of graphs in the algebraic theory of $\CSP$, or constraint satisfaction problems.  We will define them in the next section and discuss many of their nice properties. 

In the setting of reflexive graphs, $\NU$ graphs are a significant subclass of dismantlable graphs, and along the lines of \cite{BW00}, Larose, Loten, and Z\' adori gave several characterisations of them in \cite{LLZ05}.  From these characterisations
the following, which gives an  analogue of Theorem \ref{thm:BW00} for $\NU$ graphs,  is nearly immediate. 
\begin{corollary}[\cite{LLZ05}]\label{cor:LLZref}
  For a reflexive graph $H$ the following are equivalent. 
     \begin{enumerate}
       \item $H$ is an  $\NU$ graph. 
       \item $H^2$ dismantles to its diagonal. 
       \item $H$ is $\Ext$-trivial.
      \end{enumerate}   
    \end{corollary}
The definitions required for item $(2)$ are also put off  until the next section.  We show now without definitions how Corollary \ref{cor:LLZref} is indeed immediate from \cite{LLZ05} by standard arguments. The equivalence of $(1)$ and $(2)$ is explicit in \cite{LLZ05} as Theorem 1.1.  As the more general  Lemmas \ref{lem:Dis1} and \ref{lem:Dis2} give an alternate proof of the equivalence of $(3)$ with the above two characterisations, we only sketch now how this follows from Theorem 3.1 of \cite{LLZ05}.  In Theorem 3.1 the authors show that if $H$ is $\NU$, then all its $k$-subalgebras are connected. As $\Hom(G,H;p)$ is a $|V(G)|$-subalgebra of $H$, this shows that $H$ is $\Ext$-trivial. On the other hand, if $H$ is $\Ext$-trivial, then in particular $\Hom(H^2, H; p_\Delta)$ where $p_\Delta$ is the identity on the diagonal copy of $H$ in $H^2$, is connected.  From this it is immediate that there is a path of idempotents between the projections in $\Hom(H^2,H)$. This is another characterisation of reflexive $\NU$-graphs from Theorem 3.1 of \cite{LLZ05}.  

Our first goal is to generalise Corollary \ref{cor:LLZref} from reflexive graphs to all graphs. There are a couple of relevant papers that we should mention.  In \cite{LLT07} the authors generalised \cite{LLZ05} to more general relational structures. Their focus was characterisations of {\em first order definable} structures, rather than $\NU$ structures.  While these characterisations which coincide for reflexive graphs diverge somewhat for general structures, the paper still provides results that get us most of the way there.  In \cite{BBDL21}, the authors defined dismantlability for more general structures, and extended \cite{BW00} to this definition. They related a structure being dismantlable to the triviality of an extension/mixing type problem. The problem is similar to our problem $\ME(H)$, but is not the same.   

To extend Corollary \ref{cor:LLZref} to all graphs, we need to alter several of our definitions for bipartite graphs.  Indeed there are bipartite $\NU$ graphs, but there cannot be bipartite $\Ext$-trivial graphs, as for bipartite $H$, $\Hom(G,H)$ always has an even number of components.  

Calling one partite set of a bipartite graph the {\em black} vertices, and one the {\em white} vertices, we let $\Hom_B(G,H)$ be the subgraph of $\Hom(G,H)$ induced by those vertices that map black vertices of $G$ to black vertices of $H$, and let $\Hom_B(G,H;p) = \Hom(G,H;p) \cap \Hom_B(G,H)$. 

A bipartite graph $H$ is 
\begin{itemize}
  \item {\em bipartite dismantlable} if it dismantles to an edge,
  \item {\em bipartite $\Hom$-trivial} if $\Hom_B(G,H)$ is connected for all $G$, and
  \item {\em bipartite $\Ext$-trivial} if $\Hom_B(G,H;p)$ is connected for all $(G,p)$.\end{itemize}

From \cite{BW00} we get a bipartite version of Theorem \ref{thm:BW00}.

\begin{theorem}\label{thm:BW00B}
  A graph is bipartite dismantlable if and only if it is bipartite $\Hom$-trivial. 
\end{theorem}

In Section \ref{sect:ME} we prove the following.

\begin{theorem}\label{thm:mainfull}
  A  graph is an $\NU$ graph if and only if it is $\Ext$-trivial or bipartite $\Ext$-trivial. 
\end{theorem}  

 We make explicit mention here that a bipartite graph is by definition irreflexive, and that when we are talking of $\NU$ graphs, `irreflexive' and `bipartite' coincide. 

Our motivation for making Corollary \ref{cor:LLZref} explicit, and extending it to Theorem~\ref{thm:mainfull}, is that we will apply it to the problem of Shortest Path Reconfiguration.
The Shortest Path Reconfiguration problem $\SPR$ asks, for an instance $(G,u,v,P_\phi,P_\psi)$, where $u$ and $v$ are vertices of the graph $G$ and $P_\phi$ and $P_\psi$ are shortest $uv$-paths in $G$, if one can get from $P_\phi$ to $P_\psi$ by a sequence of shortest $uv$-paths in $G$ by changing one vertex at a time.

 In \cite{Bonsma13} it was shown that $\SPR$ is in $\PSC$, and in \cite{Bonsma13} and \cite{GKL} it is shown the problem drops to $\Poly$ when the instance $G$ is restricted to various classes of graphs.  In particular, where again a graph $G$ is {\em $\SP$-trivial} if every instance $(G,u,v,P,Q)$ of $\SPR$ is a YES-instance, it is shown in \cite{Bonsma13} that chordal graphs are $\SP$-trivial, and in \cite{GKL} that grids and bridged graphs are $\SP$-trivial. 

To relate this back to $\RE(G)$, we note that for $\RH(H)$, one often replaces the graph $\Hom(G,H)$ with the subgraph $\Col(G,H)$ on the same vertices but having only those edges of $\Hom(G,H)$ on which the endpoint maps differ on a single vertex.  It is well known and easy to show that when $G$ is a graph (as opposed to a directed graph) there is a path between vertices of $\Hom(G,H)$ if and only if there is a path between them in $\Col(G,H)$. So $\Hom(G,H)$ and $\Col(G,H)$ have the same components and can both be used as the reconfiguration graph for $\RH(H)$.  Similarly the components of $\Hom(G,H;p)$ and of the subgraph $\Col(G,H;p) = \Col(G,H) \cap \Hom(G,H;p)$ are the same.

With this, one sees that an instance $(G,u,v,P_\phi,P_\psi)$ of $\SPR$ is just an instance $(P, p, \phi, \psi)$ of $\RE(G)$ where $P$ is a path of length $d(u,v)$, $p$ is the $G$-precolouring that fixes its endpoints to $u$ and $v$, and $\phi$ and $\psi$ are maps of $P$ onto $P_\phi$ and $P_\psi$ respectively. The {\em shortest path reconfiguration graph} $\SP(G,u,v)$ is $\Col(P,G;p)$. 

So we get from Theorem \ref{thm:mainfull} that $G$ is $\SP$-trivial if it is an $\NU$ graph. Indeed, as the existence of loops on $G$ has no effect on $\SPR$, we get that the following.

 \begin{corollary}\label{cor:SPR}
     If $G$ is an $\NU$ graph then $G$, with loops removed, is $\SP$-trivial.
\end{corollary}

 Chordal graphs, and most bridged graphs, are reflexive $\NU$ graphs;  grids  are irreflexive $\NU$ graphs.   So while our result does not recover all the results of these two papers about $\SP$-trivial graphs, it recovers most of them, and because it is done not only for reflexive or irreflexive graphs, greatly extends the set of known $\SP$-trivial graphs. 

In \cite{KMM11} it was shown that components of the reconfiguration graph $\SP(G,u,v)$ could have diameter exponential in $d = d(u,v)$.  In \cite{Bonsma13} and \cite{GKL}, when a graph $G$ was shown to be $\SP$-trivial, the diameter of the reconfiguration graph $\Col(G,u,v)$ was shown to have much smaller diameter. To get similar bounds on this diameter when we apply Theorem \ref{thm:mainfull} we bound the diameter of $\Hom(G,H;p)$ for $\NU$ graphs $H$.  This is done in Corollary \ref{cor:diam2} and then in  Corollary \ref{cor:polynomial} it is used to show that there is a parameterised polynomial time algorithm for finding the shortest path in $\Hom(G,H;p)$ between two vertices $\phi$ and $\psi$.  

The bound on the diameter of $\Hom(G,H;p)$ gives a rough bound on the diameter of $\Col(G,H;p)$.  In Section \ref{sect:ReconLen} we sharpen the bound on the diameter of $\Col(G,H;p)$ in the case that $H$ is a $3$-$\NU$ graph, and give an infinite family of examples for which this bound is sharp.  In doing so, we introduce the notion of an `efficient dismantling' of a dismantlable graph, and show that $3$-$\NU$ graphs have efficient dismantlings.  This may be of independent interest. 

In Section \ref{sect:SPR} we apply Theorem \ref{thm:mainfull}, and our bounds on the diameter of $\Col(G,H;p)$ to the shortest path reconfiguration problem. These are not quite as good as the bounds for chordal graphs and grids from \cite{Bonsma13} and \cite{GKL}, but are much more general.

\section{Background results and definitions}\label{sect:bg}

The {\em categorical product} $H_1 \times H_2$ of two graphs $H_1$ and $H_2$ is the graph with vertex set $V(H_1) \times V(H_2)$ in which $(v_1,v_2) \sim (v_1',v_2')$ if $v_i \sim v_i'$ for each $i$.  It is well known and easy to check that the projections $\pi_i: H_1 \times H_2 \to H_i: (v_1,v_2)\mapsto v_i$ are homomorphisms. The product is clearly commutative and associative, so one can write $H^2$ for $H \times H$ and $H^n$ for $H \times H^{n-1}$.

   A homomorphism $r: H \to R$ of $H$ to a subgraph $R$ is a {\em retraction} if it restricts to the identity map $\id$ on $V(R)$.  If such $r$ exists, then $r(H) = R$ is a {\em retract} of $H$.   
  The following is well known, showing that the class of \kNU{} graphs is what is called a graph variety. See, for example, \cite{NU2}.
  
  \begin{fact}\label{fact:NUvar}
    The class of \kNU{} graphs is closed under products and retractions. 
  \end{fact}

  \subsection{Alternate definition of the Hom-graph}

  The $\Hom$-graph $\bHom(G,H)$ defined in the introduction has an alternate and equivalent definition that is useful for topological arguments. Let $I_\ell$ be the reflexive path $0 \sim 1 \sim \dots, \sim \ell$ of length $\ell$. For a homomorphism $h:I_\ell \times G \to H$ and $i \in V(I_\ell)$ the {\em $i^{th}$ step of $h$} is the homomorphism $h_i:G \to H$ that we get by restricting $h$ to the copy $i \times G$ of $G$ in $I_\ell \times G$ that is induced on the vertex set $\{(i,g) \mid g \in V(G)\}$. 
  
  \begin{definition}
  For $H$ and $(G,p)$, the $\Hom$-graph $\Hom(G,H;p)$ is the graph on the $H$-colourings of $G$ extending $p$ in which two extensions $\phi$ and $\psi$ are adjacent if there is a homomorphism $h: I_1 \times G \to H$ such that  $h_0 = \phi$ and $h_1 = \psi$.  
  \end{definition}

  A walk from $\phi$ to $\psi$ in $\bHom(G,H;p)$ corresponds exactly to a homomorphism $h: I_\ell \times G \to H$ such that $h_0 = \phi$ and $h_\ell = \psi$ and each $h_i$ extends $p$.   With this and the simple observation that for homomorphisms $\phi, \phi': G \to H$ the map 
  \[ \phi \times \phi': G \to H^2: (a,b) \mapsto (\phi(a),\phi(b)) \] 
 is also homomorphism, it is easy to verify such  facts as the following. 
  
\begin{fact}\label{fact:times}
  For walks $h,h': I_\ell \times G \to H$ in $\bHom(G,H;p)$
  the following is a walk in  $\bHom(G,H^2;p \times p)$: 
   \[ h \times h' := (h_0 \times h'_0) \sim (h_1 \times h'_1) \sim  \dots \sim  (h_\ell \times h'_\ell). \] 
 \end{fact}

   \subsection{Dismantling, the bipartite resolution, and symmetric shadows}
   
   Recall that a vertex $v'$ {\em dominates} a vertex $v$ in a graph $H$ if $N(v) \subset N(v')$. Note that if $v$ is reflexive, then $v \in N(v)$. 
   If $v'$ dominates $v$ then there is a retraction $d: H \to H \setminus \{v\}$ taking $v$ to $v'$. This is a {\em dismantling retraction}; we say it {\em dismantles $v$}. A graph $H$ dismantles to $R$ if there is a sequence of dismantling retractions $d_1, d_2 \dots, d_\ell$, called a {\em dismantling}, whose composition $\bar{d_\ell}:= d_\ell \circ d_{d\ell-1} \circ \dots d_1$ is a retraction of $H$ to $R$. 
   
   The following is clear and proved in more generality as Lemma 5.1 of \cite{LLT07}. It allows one to greedily decide if a graph dismantles to a given subgraph.  
   
   \begin{fact}\label{fact:dismantreorder}
     If $d_1, d_2, \dots, d_\ell$ is a dismantling of $H$ to $R$, such that $d_i$ dismantles the vertex $v_i$,  and $d'_i$ is a dismantling retraction of $H$ to $H \setminus \{v_i\}$, then \[ d_i',d_i'\circ d_1, \dots, d'_i \circ d_{i-1}, d_{i+1}, \dots, d_\ell \]  is also a dismantling of $H$ to $R$. 
   \end{fact}
    
   For a graph $H$, the {\em bipartite resolution of $H$} is the bipartite (so irreflexive) graph $B(H):= K_2 \times H$.  As $K_2$ has \kNU{} polymorphisms for all $k \geq 3$, the following is immediate from Fact \ref{fact:NUvar}.
   
   \begin{fact}\label{fact:BNU}
    If $H$ is a \kNU{} graph then so is $B(H)$. 
   \end{fact}

   The following was proved in \cite{BFH93} in the case that $H$ is reflexive and $R$ is a single vertex subgraph, but the proof gives the following.
   We recall the proof as we will use notions from it later.    
    \begin{lemma}\cite{BFH93} \label{lem:disred}
       A graph $H$ dismantles to a subgraph $R$ if and only if 
       $B(H)$ dismantles to $B(R)$.
    \end{lemma}
    \begin{proof}
       Let $d_1, d_2, \dots, d_\ell$ be a dismantling of $H$ to $R$ where $d_i$ dismantles the vertex $v_i$,  and let $R_i = \bar{d}_i(H)$.  Let $a_i: B( R_i) \to B(R_i)$ be the map that dismantles $(0,v_i)$ to $(0,d_i(v_i))$, and let $b_i$ be the map that dismantles $(1,v_i)$ to $(1,d_i(v_i))$. Their composition takes $B(R_i)$ to $B(R_{i+1})$, and so 
       \begin{equation}\label{eq:BipResDis}
         a_1, b_1, a_2, b_2, \dots, a_\ell, b_\ell
       \end{equation}  
       is a dismantling of $B(H)$ to $B(R)$.
       
       On the other hand let $d_1, d_2, \dots, d_{2\ell}$ be a dismantling of $B(H)$ to $B(R)$.  Half of these maps dismantle vertices of the form $(0,v)$. Where the $i^{th}$ such map dismantles $(0,v_i)$ to $(0,v'_i)$, let $s_i$ be the map defined on $\bar{s}_{i-1}(H)$ that dismantles $v_i$ to $v_i'$. Then 
       \begin{equation}\label{eq:ShadDis}
        s_1, s_2, \dots, s_\ell  
       \end{equation} 
       is a dismantling of $H$ to $R$.
      \end{proof} 
      
      \begin{definition}\label{def:BipResAndShad}
        Given a dismantling $d_1, \dots, d_\ell$ of $H$ to $R$, the  dismantling \eqref{eq:BipResDis} of $B(H)$ to $B(R)$ is its
        {\em bipartite resolution}.  
        Given a dismantling $d_1, \dots d_{2\ell}$, the dismantling 
        \eqref{eq:ShadDis} is its {\em symmetric shadow}.
      \end{definition}  
      
   Basic to \cite{BW00} was the observation that a dismantling $d_1, \dots, d_\ell$ of $H$ to $R$ induces a path   
   \begin{equation*}
    \id, d_1, \bar{d}_2, \dots, \bar{d}_\ell
   \end{equation*}    
   from the identity $\id: H \to H$ to the retraction $\bar{d}_\ell$ of $H$ to $R$ in 
   $\bHom(H,H)$. 
   The following extension of this idea for reflexive graphs is Lemma 2.5 of \cite{LLZ05} and for more general graphs is Lemma 5.3 of \cite{LLT07}.
   
   \begin{lemma}\label{lem:dismant}
      Let $R$ be a subgraph of a graph $H$ and $\id_{R}$ be the $H$-precolouring of $H$ that is the restriction to $R$ of the identity map $\id: H \to H$. The graph $H$ dismantles to $R$ if and only if there is a path in $\bHom(H,H;\id_{R})$ from $\id$ to a retraction $r:H \to R$.
   \end{lemma}
    
    This yields the following useful corollary by observing that a graph $H$ retracts to any looped vertex, and a bipartite graph $H$ retracts to any edge.  
   \begin{fact}\label{fact:dismant}
     If $H$ is $\Ext$-trivial, then it dismantles to any looped vertex.
     If $H$ is bipartite $\Ext$-trivial, then it dismantles to any edge. 
   \end{fact}

    \subsection{Critical tree obstructions}

      A {\em critical obstruction} $(T,p)$ for a graph $H$ consists of a graph $T$ and an $H$-precolouring $p$ of $T$ such that  no homomorphism $\phi: T \to H$ extends $p$, but for any proper subgraph $T'$ of $T$ there exists a homomorphism $\phi: T' \to H$ extending $p$ (or its restriction to $T'$).
      
      It is well known (see \cite{Zad97}) that a graph has a \kNU{} polymorphism if and only if it has no critical obstructions $(T,p)$ with $|p| = k$.  A critical obstruction $(T,p)$ is a 
      {\em critical tree obstruction} if $T$ is a tree and $p$ is defined on its set of leaves. 
      In \cite{LLZ05} it was shown for reflexive graphs that $H$ has a \kNU{} polymorphism if and only if it has no critical tree obstructions with $k$ leaves.  The same was shown for irreflexive graphs in \cite{NU2} using a characterisation of structures with tree duality from \cite{LLT07}.

   \section{Characterisations of ${\rm Ext}$-trivial and bipartite ${\rm Ext}$-trivial graphs}\label{sect:ME}

    The {\em diagonal} $\Delta = \Delta(H^2)$ of $H^2$ is the subgraph isomorphic to $H$ induced on the vertex set $\{ (x,x) \mid x \in V(H)\}$. 
    If $H$ is bipartite then $H^2$ is disconnected and cannot dismantle to $\Delta$.
    The {\em diagonal component} $C_\Delta(H^2)$ of $H^2$ is the component of $H^2$ containing $\Delta$. 
    
  Note that unless $H$ is bipartite, $C_\Delta(H^2) = H^2$. Letting $\Hom_B(G,H;p)$ be just $\Hom(G,H;p)$ when $H$ is not bipartite will allow us to unify the bipartite and non-bipartite cases of many statements and some proofs.  
 In this section we show the following.

  \begin{theorem}\label{thm:LLZ1} For a finite connected graph $H$, the following are equivalent.
  \begin{enumerate}
      \item The graph $H$ is an $\NU$ graph. 
      \item The graph $C_\Delta(H^2)$ dismantles to $\Delta(H^2)$. 
      \item The graph $H$ is $\Ext$-trivial or bipartite $\Ext$-trivial.  
  \end{enumerate}
  \end{theorem}  
  We break the proof of Theorem \ref{thm:LLZ1} into two subsections.   The equivalence of items $(2)$ and $(3)$ is proved in Subsection \ref{sub:charDism}, and that of $(2)$ and $(1)$ is proved in Subsection \ref{sub:charNU}.  
  Before moving on to the proof of Theorem \ref{thm:LLZ1} we observe some useful corollaries.

  In \cite{LLT07} it was shown that finite tree duality for a structure $H$ is characterised by the fact that $H^2$ dismantles to a retraction of its diagonal. Thus we get the following from Theorem \ref{thm:LLZ1}, extending a result proved given for reflexive graphs in \cite{LLZ05} and for irreflexive graphs in \cite{NU2}. 
  
    \begin{corollary}\label{cor:treeDuals}
     A graph $H$ admits a \kNU{} polymorphism if an only if it has no critical tree obstructions with $k$ or more leaves. 
    \end{corollary}
    
    Observing for a non-bipartite graph that a critical tree obstruction in $B(H)$ translates to a critical tree obstruction in $H$, we get the following converse of Fact \ref{fact:BNU}. This fact, for conservative $\NU$ polymorphisms, is known to follow from \cite{FHH}, but is far from trivial.
  
  \begin{corollary}\label{cor:BNU}
    A non-bipartite graph $H$ has an $\NU$-polymorphism if and only if $B(H)$ does. 
  \end{corollary}

   \subsection{Characterisation by dismantling}\label{sub:charDism}
   
   In this section we prove the equivalence of  items $(2)$ and $(3)$ of Theorem \ref{thm:LLZ1}.  The proof follows the ideas of the proof from \cite{LLZ05} for reflexive graphs that we sketched in the introduction, only we tailor their proofs to include our definition of $\Ext$-trivial.  Besides this definition, our main contributions in this section are observing that the proofs work on general graphs using $C_\Delta(H^2)$ where they used $H^2$,  and paying attention to the lengths of the paths involved.  

 The length of a path is the number of edges. Recall that if $H$ is not bipartite, then $\bHom_B(G,H;p)$ is just $\bHom(G,H;p)$.

    \begin{lemma}\label{lem:Dis1}
       If $C_\Delta(H^2)$ dismantles to $\Delta(H^2)$ then for any $\phi$ and $\psi$ in $\bHom_B(G,H;p)$ there is a walk  from $\phi$ to $\psi$ in $\bHom_B(G,H;p)$, and so $H$ is $\Ext$-trivial or bipartite $\Ext$-trivial.  This walk has length at most $2(n_H^2 -  n_H)$ and at most $n_H^2 -2n_H-2$ if $H$ is bipartite.
    \end{lemma}
    \begin{proof}
     Assume that there is a dismantling of $C_\Delta(H^2)$ to $\Delta = \Delta(H^2)$. By Lemma \ref{lem:dismant} this corresponds to a walk $r_0, \dots, r_d$ from the identity $r_0 = \id$ to a retraction $r_d$ to $\Delta$ in $\bHom_B(C_\Delta(H^2),C_\Delta(H^2); \id_\Delta)
     \leq \bHom_B(H^2,H^2;\id_\Delta)$. 
     Defining the map $a_i$ from $G$ to $H$ by 
     \[ a_i = \pi_1 \circ r_i \circ (\phi \times \psi),  \]   
     the sequence $a_0, a_1 \dots, a_d$ is a walk in $\bHom_B(G,H)$ from $\phi$ to $a_d$. For a vertex $v$ in the support of $p$ we get   
       \[ a_i(v) = \pi_1 \circ r_i \circ (\phi \times \psi)(v)
                 = \pi_1 \circ r_i((p(v),p(v))) = \pi_1((p(v),p(v))) = p(v) \]
     showing that this walk is indeed in $\bHom_B(G,H;p)$.  
     
     Similarly the sequence $b_d, \dots, b_0$ defined by 
       \[ b_i = \pi_2 \circ r_i \circ (\phi \times \psi)  \]
     is a walk in $\bHom_B(G,H;p)$ from $b_d$ to $\psi$.
     
     But $a_d = b_d$, as $r_d$ maps to $\Delta$ so maps under $\pi_1$ and $\pi_2$ to the same place.  Thus we have a walk 
     \begin{equation}\label{eq:abwalk}
         a_0, a_1, \dots a_d,b_{d-1}, \dots, b_0
     \end{equation} of length $2d$ from $\phi$ to $\psi$ in $\bHom_B(G,H;p)$. 

     In the case that $H$ is bipartite we can improve this a little.  Indeed, as $r_d$ dismantles some final vertex $(a,b)$ to the diagonal, we have that $r_{d-1}(C_\Delta(H^2)) = \Delta \cup \{(a,b)\}$. Thus $a_{d-1}$ and $b_{d-1}$ differ only on vertices $v$ such that  $r_{d-1} \circ (\phi \times \psi)(v) = (a,b)$. For such vertices $v$ we have $a_{d-1}(v) = a$ and $b_{d-1}(v) = b$.  Now, for any neighbour $u$ of such a $v$, $r_{d-1} \circ (\phi \times \psi)(u)$ is on the diagonal, so $a_{d-1}(u) = b_{d-1}(u)$ and so we get $a_{d-1}(v) \sim a_{d-1}(u) = b_{d-1}(u)$.   Thus $a_{d-1} \sim b_{d-1}$ and we have a walk

     \begin{equation}\label{eq:abwalkB}
         a_0, a_1, \dots a_{d-1},b_{d-1}, \dots, b_0
     \end{equation} of length $2d-1$ from $\phi$ to $\psi$ in $\bHom_B(G,H;p)$. 
     
     As $C_\Delta(H^2)$ has $n^2_H$ vertices if $H$ is non-bipartite (half of this if $H$ is bipartite), it dismantles to $\Delta$ in  $d = n_H^2 - n_H$ (respectively $d = \frac{1}{2}n_H^2 -n_H$) steps. 
     Using these $d$ in  \eqref{eq:abwalk} and \eqref{eq:abwalkB}, we get the claimed length calculations.
    \end{proof}

      The converse implication is even easier. 
    
    \begin{lemma}\label{lem:Dis2}
       If $H$ is (bipartite) $\Ext$-trivial then $C_\Delta(H^2)$ dismantles to $\Delta(H^2)$. 
    \end{lemma}
    \begin{proof}
       
     Assume that $H$ is (bipartite) $\Ext$-trivial, so $\bHom_B(G,H;p)$ is connected for all $(G,p)$.  Where $\pi_1: C_\Delta(H^2) \to H$ is the restriction of the first projection and $i: H \to H \times H: v \mapsto (v,v)$ is the diagonal injection, the restriction $p$ of $i \circ \pi_1$ to $\Delta(H^2)$ is the identity. 
     Thus there is a path in $\bHom_B(C_\Delta(H^2),\Delta(H^2);p)$ from $\id$ to $i \circ \pi_1$, which by Lemma \ref{lem:dismant} gives a
     dismantling of $C_\Delta(H^2)$ to $\Delta(H^2)$. 
    \end{proof}
    
    We have proved, with these lemma, the equivalence of items $(2)$ and $(3)$ of Theorem \ref{thm:LLZ1}.

   \subsection{Characterisation by ${\rm NU}$-polymorphisms}\label{sub:charNU}

   In this subsection we finish the proof of Theorem \ref{thm:LLZ1} by showing the implications $(2) \Rightarrow (1)$ and $(1) \Rightarrow (3)$.

   We start with $(2) \Rightarrow (1)$. It is the harder implication, but it is mostly done elsewhere.  Indeed, as we have already seen that it holds for bipartite graphs by \cite{LLZ05}, it is enough to show that if $H^2$ dismantles to its diagonal, then $H$ has an $\NU$ polymorphism. This is essentially Theorem 5.7 and Corollary 4.5 of \cite{LLT07}.
   Indeed Theorem 5.7 of \cite{LLT07} tells us that if $H^2$ dismantles to its diagonal, then it has a first order definable CSP, and Corollary 4.5 tells us that if  it is a core with a first order definable CSP, then it has an $\NU$ polymorphism.
   The implication thus follows by the well known fact, stated for example in Lemma 4.3 of \cite{NU2}, that a graph has an $\NU$ polymorphism if and only if the core structure one gets from it by adding all unary singleton relations does.

  So we have only to show the implication $(1) \Rightarrow (3)$ of Theorem \ref{thm:LLZ1} which says that if $H$ is $\NU$ then it is $\Ext$-trivial.  Well this already holds if $H$ is bipartite, as in this case we have the implication $(1) \Rightarrow (2)$ from from \cite{NU2}, and in the previous subsection we showed the equivalence of $(2)$ and $(3)$.  We will use this to get the implication $(1) \Rightarrow (3)$ for non-bipartite $H$.   Indeed we show first that if $H$ is $\NU$, then it is dismantlable; we then use dismantlability and the $\NU$ polymorphism to get that $H$ is $\Ext$-trivial.

  \begin{lemma}\label{lem:disloop} 
      Any bipartite $\NU$ graph dismantles to any edge. 
      Any non-bipartite $\NU$ graph dismantles to any looped vertex.
  \end{lemma}
  \begin{proof}
     The first statement is immediate from Fact  \ref{fact:dismant} as we have already proved that bipartite $\NU$ graphs are bipartite $\Ext$-trivial.  So let $H$ be a non-bipartite graph with an $\NU$ polymorphism.  The bipartite graph $B(H)$ also has an $\NU$ polymorphism by Fact \ref{fact:BNU}, so by the first statement of the lemma, $B(H)$ dismantles to any edge. In particular it dismantles to the edge $(0,v)(1,v)$ for any looped vertex $v$ of $H$.  By Lemma \ref{lem:disred} we get that $H$ dismantles to $v$.
  \end{proof}

   Let $H$ be any $\NU$ graph.  By Lemma \ref{lem:disloop} it is dismantlable or bipartite dismantlable.  By Theorem \ref{thm:BW00} or \ref{thm:BW00B} we therefore get that $\Hom_B(G,H)$ is connected for any $G$.  The following then shows that $\Hom_B(G,H;p)$ is connected for any $(G,p)$ completing the proof of implication $(1) \Rightarrow (3)$, and so the proof of  Theorem \ref{thm:LLZ1}.  
   
   \begin{proposition}\label{prop:discrel1}
    Let $H$ be a connected \kNU~ graph, and let $\phi$ and $\psi$ be vertices of $\bHom_B(G,H;p)$ for some $H$-precolouring $p$ of $G$. If there is a walk of length $\ell$ between $\phi$ and $\psi$ in $\bHom_B(G,H)$, then there is a walk of length $(k-2)\ell$ between them in $\bHom_B(G,H;p)$.
  \end{proposition}
  \begin{proof}
  For $\phi$ and $\psi$ in $\bHom_B(G,H;p)$, let $h^* = h^*_1,h^*_2, \dots, h^*_\ell$ be a walk from $\phi$ to $\psi$ in $\bHom_B(G,H)$. We construct a walk $h$ from $\phi$ to $\psi$ in $\bHom_B(G,H;p)$.
      Indeed, where $f:H^k \to H$ is the \kNU~ polymorphism, 
     for $j = 1, \dots, k-2$ and $i = 1, \dots, \ell$ let 
     \[h_{j\cdot \ell + i} = f \circ (\underbrace{\psi \times \psi \times \dots \times \psi}_{j \mbox{ copies}} \times h^*_i \times \underbrace{\phi \times \dots \times \phi}_{k-j-1}). \]
     
     As $h^*_0 = \phi$ and $h^*_\ell = \psi$ this is a walk in $\bHom_B(G,H)$,  as $f$ is \kNU, it goes from \[f \circ (\psi \times \phi \times \dots \times \phi) = \phi \qquad \mbox{ to } \qquad f  \circ (\psi \times \dots \times \psi \times \phi) = \psi.\]
     On vertices $v$ in the support of $p$, $\phi(v) = \psi(v) = p(v)$, so all of these functions map $v$ to $p(v)$ by the fact that $f$ is \kNU.
     Thus we have a path of length $(k-2)\ell$ in $\bHom_B(G,H;p)$, as needed. 
  \end{proof}

    This completes the proof of Theorem \ref{thm:LLZ1}.

    \subsection{Diameter of $\bHom(G,H;p)$ and finding shortest paths}

    Our proofs give two bounds on the diameter of $\bHom_B(G,H;p)$ when $H$ is (bipartite) $\Ext$-trivial.

   \begin{corollary}\label{cor:diam2}
    For a \kNU{} graph $H$ with $n_H$ vertices, the graph $\bHom_B(G,H;p)$ has diameter at most $2n_H\min((k-2),n_H)$ or at most  $n_H\min(2(k-2),n_H)$ if $H$ is bipartite. 
   \end{corollary}
   \begin{proof}
     The bounds of $2n_H^2$ and $n_H^2$ on the diameter are given explicitly in Lemma~\ref{lem:Dis1}. The bound of $2n_H(k-2)$ comes from the bound $\ell(k-2)$ on the length of the walk in Proposition \ref{prop:discrel1} by observing that if $H$ dismantles to a vertex or edge, then it does so in at most $n_H-1$ steps, and so there is a path of length $\ell  < 2n_N$ between any two maps in $\bHom_B(G,H)$. 
   \end{proof} 
   
    While $(k-2)$ is smaller than $n_H$ in any examples I know, it is not clear that this is always true. Indeed, it is shown in \cite{BartoDrag}, see also \cite{ZhukNU}, that there are digraphs on $n$ vertices for which the smallest $\NU$ polymorphism have arity doubly exponential in $n$. We do not know if such symmetric graphs exist.

     It is known from \cite{FV99} that for a graph $H$ admitting a \kNU{} polymorphism, an instance $G$ of the problem $\Hom(H)$ can be solved in time $O(n_G^kn_H^k)$. For fixed $k$, this is polynomial in the order $n_G$ of the instance. A list version of the problem can be solved in this time, too.  
     This, with Corollary \ref{cor:diam2}, allows us to find a shortest path in $\bHom(G,H;p)$ between two vertices $\phi$ and $\psi$ as follows.

     For $\ell = 2, \dots, 2(k-2)n_H=:L$ solve the instance $\bHom(I_\ell \times G,H;q)$ where the $H$-precolouring $q$ of  $I_\ell \times G$ is defined so that $q_0 = \phi$, $q_\ell = \psi$, and for all $i \in [\ell-1]$ and all vertices $g$ of $G$ in the support of $p$, $q_i(g) = p(g)$.  A solution, the existence of which we can determine in time polynomial in $n_G\cdot \ell$, so polynomial in $n_G$, yields a path from $\phi$ to $\psi$. If there is no solution, there is no such path  of length $\ell$ or shorter.  
     By Corollary \ref{cor:diam2} we will find a path by the time $\ell$ reaches $L$. This is polynomial in $n_G$. We record this. 
     
     \begin{corollary}\label{cor:polynomial}
      For fixed $k$, and \kNU{} graph $H$, there is a polynomial time algorithm for finding the shortest path between $\phi$ and $\psi$ for any instance $(G,p,\phi,\psi)$ of $\RE(H)$.     
     \end{corollary}
     
     Though this algorithm is  polynomial in $n_G$ for fixed parameter $k$, it is exponential in $k$, and $k$ can grow exponentially in the size of $n_H$. A polynomial algorithm independent of $k$ would be quite interesting.  
     
    
     
    \section{Bounds on recolouring length}\label{sect:ReconLen}
    
   As mentioned in the introduction, the reconfiguration graph for $H$-extensions is often taken to be the `recolouring' subgraph $\bCol(G,H;p)$ of $\bHom(G,H;p)$ consisting of edges whose endpoints are homomorphisms differing on a single vertex.  If there is a edge between vertices $\phi$ and $\psi$ in $\bCol(G,H;p)$ then one can {\em resolve} it to a path of  edges of $\Col(G,H;p)$ simply by changing the vertices $v$ on which $\phi$ and $ \psi$ differ one at a time from $\phi(v)$ to $\psi(v)$. So there is a path between vertices $\phi$ and $\psi$  in $\bCol(G,H;p)$ if and only if there is one in $\bHom(G,H;p)$, though these paths generally have different length. For a path between vertices $\phi$ and $\psi$, we specify that it is a path in $\bCol(G,H;p)$ by calling it a {\em recolouring-path}.  The length of the shortest $(\phi,\psi)$-recolouring path is the {\em recolouring distance} $d_{\bCol}(\phi,\psi)$. 

 From Corollary \ref{cor:diam2} we get a trivial upper bound of 
   \begin{equation}\label{eq:reconbound}
     d_{\bCol}(\phi,\psi)  \leq  2n_Gn_H 
   \end{equation}
 on the recolouring distance for any instance $(G,p,\phi,\psi)$ of $\RE(H)$ where $H$ is \kNU{}. In this section, we improve this when $H$ is \kNU[3].

 To begin with, observe that there is also a trivial lower bound on the recolouring distance of two colourings. We say that a {\em recolouring step}, or an edge $\phi\phi'$ in $\bCol(G,H)$, {\em moves} a vertex $v$ if $\phi'(v) \neq \phi(v)$.  A reflexive vertex $v$ of $G$ must move at least $d_H(\psi(v),\phi(v))$ times as we recolour $\phi$ to $\psi$; an irreflexive vertex (with an edge) must move at least half this many times. So for $\phi$ and $\psi$ in $\bCol(G,H;p)$ we clearly get the following where $V_\ell$ and $V_u$ are respectively the sets of looped and unlooped vertices of $G$,
   \begin{equation}\label{eq:reconbound}
        \sum_{v \in V_\ell} d_H(\phi(v),\psi(v))  + \sum_{v \in V_u} \left\lceil \frac{1}{2}d_H(\phi(v), \psi(v)) \right\rceil \leq d_{\bCol}(\phi,\psi). 
  \end{equation}  
   
   Writing $|\psi - \phi|$ for $\sum_{v \in V(G)} d_H(\phi(v),\psi(v))$ and $\odd(\phi,\psi)$ for the number of vertices $v$ of $G$ for which $d_H(\phi(v), \psi(v))$ is odd,
   this specialises to the following.  

   \begin{fact}\label{fact:lowerbound}
     Let $G$ be a non-empty connected graph, and $\phi$ and $\psi$ be maps in the same component of $\bCol(G,H;p)$. 
     If $G$ is reflexive, then 
             \[ |\phi - \psi| \leq d_{\bCol}(\phi,\psi) \] 
        and if $G$ is bipartite, then
          \[ \frac{1}{2}( |\phi - \psi| + \odd(\phi,\psi) |) \leq d_{\bCol}(\phi,\psi). \] 
    \end{fact}

    Our main result in this section gives  complementary upper bounds, which at least in the cases that $G$ is reflexive or bipartite, are less than twice the  lower bounds. 

   \begin{theorem}\label{thm:recoldistbound}
     For any \kNU[3] graph $H$, any $H$-precolouring $p$ of a graph $G$, and any maps $\phi$ and $\psi$ in $\bCol(G,H;p)$, we have
         \[  d_{\bCol}(\phi,\psi) \leq |\phi - \psi| + \odd(\phi,\psi) - 1; \]
    in particular, this holds if $G$ is reflexive. 
    If $G$ is bipartite, then $H$ must be bipartite, and we have 
         \[  d_{\bCol}(\phi,\psi) \leq |\phi - \psi| - 1. \]
   \end{theorem}

   The proof of Theorem \ref{thm:recoldistbound} will take some work, and we put this off until the last subsection of this section.  Before we get to it, we give examples showing that these bounds are sharp. 

  \subsection{Sharpness examples for Theorem \ref{thm:recoldistbound}}

    \begin{figure}
     \begin{center}
    \begin{tikzpicture}[every node/.style = {vert}]
    
      \foreach \i in {0,1,2,3,4,5}{\foreach \j in {0,1,2,3,4,5}{
      \draw (\i,\j) node (v\i\j){};}}  
      \foreach \i[evaluate=\i as \ip using {int(\i+1)}] in {0,1,2,3,4}{\foreach \j[evaluate=\j as  \jp using {int(\j+1)}] in {0,1,2,3,4}{
      \draw[gray] (v\i\j) edge (v\i\jp) edge (v\ip\jp) edge (v\ip\j);
      \draw[gray] (v\i\jp) edge (v\ip\j);
      }}
      \draw[gray, thin] (v55) edge (v05) edge (v50);
      
      \draw[line width=1mm, blue] (v00) -- (v01) -- (v12) -- (v23) -- (v34) -- (v45) -- (v55) ; 
      \draw[line width=1mm, red] (v00) -- (v10) -- (v21) -- (v32) -- (v43) -- (v54) -- (v55) ; 
      
      \Vlabel[6mm]{v00}{left}{(0,0)}
      \Vlabel[4mm]{v10}{below}{(1,0)}
      \Vlabel[6mm]{v01}{left}{(0,1)}
      \Vlabel[6mm]{v55}{right}{(5,5)}
      \Vlabel[5mm]{v23}{above left}{\phi}
      \Vlabel[5mm]{v32}{below right}{\psi}

     \begin{scope}[xshift=8cm]
        \foreach \i in {0,1,2}{\foreach \j in {0,1,2,3,4,5}{ \draw (\i,\j) node (u\i\j){};}}
        \foreach \i in {0,1,2}{\foreach \j in {0,1,2,3,4,5}{ \draw[gray] (u0\j) -- (u1\j) -- (u2\j);
                                                                                    \draw[gray] (u\i0) -- (u\i1) --(u\i2) --(u\i3) --(u\i4) --(u\i5);  }}
      \draw[line width=1mm, blue] (u10) -- (u00) -- (u01) -- (u02) -- (u03) -- (u04) -- (u05) -- (u15) ; 
      \draw[line width=1mm, red]  (u10) -- (u20) -- (u21) -- (u22) -- (u23) -- (u24) -- (u25) -- (u15) ; 
      \Vlabel[5mm]{u02}{above left}{\phi}
      \Vlabel[5mm]{u22}{above right}{\psi}

     \end{scope}
    \end{tikzpicture}\caption{Pairs of paths $\phi$ and $\psi$ with recolouring distance \\ $|\phi - \psi| + \odd(\phi,\psi) - 1$ (left)  and $|\phi-\psi|-1$ (right)}\label{fig}
    \end{center}
    \end{figure}
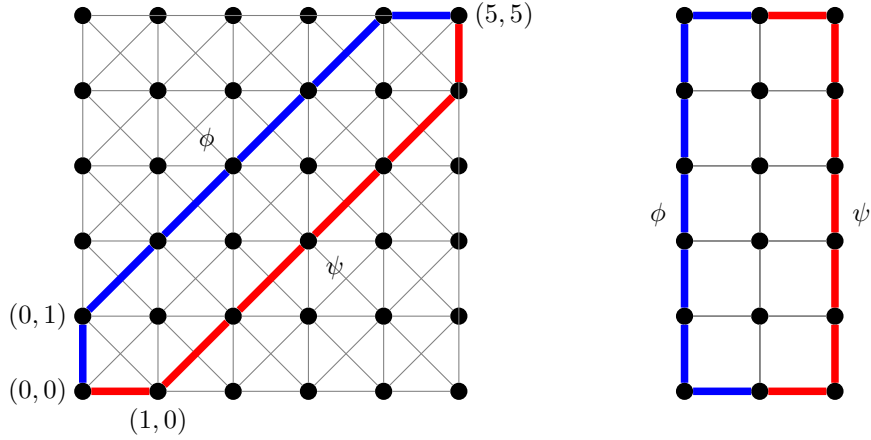

    \begin{example}
        Let $H$ be the product of two reflexive paths of length $n$. It has vertex set $\{0,1, \dots, n\}^2$, and $(i,j) \sim (i',j')$ if $|i' - i| \leq 1$ and $|j'-j| \leq 1$. (See Figure \ref{fig} for the case $n = 5$; all vertices are looped, but the loops are not shown.) $H$ is a \kNU[3] graph.  Where $G$ is an irreflexive path of length $n+1$, 
        let $\phi$ map it to the path
          \[ (0,0), (1,0), (2,1), \dots, (n,n-1), (n,n) \]
        and  $\psi$ map it to 
          \[ (0,0), (0,1), (1,2), \dots, (n-1, n), (n,n). \]
        
        Both are extensions of $p$ mapping the endpoints of $G$ to $(0,0)$ and $(n,n)$.
        One can check that $|\psi - \phi| = n-1$ and that $\odd(\phi,\psi) = n-1$.  The only vertices of $\phi(G)$ that can move are 
        $(1,0)$ and $(n,n-1)$, and they can only move to $(1,1)$ and $(n-1,n-1)$ respectively.  Without loss of generality, move $(1,0)$ to $(1,1)$.  The vertex $(2,1)$ now becomes free to move to $(2,2)$.  One can check that no vertex of $\phi(G)$ can move below the line 
        \[ (0,0), (1,1), \dots, (n-1,n-1), (n,n) \]
        until all but one of them have moved onto it, and then that last vertex $v$ can move directly from $\phi(v)$ to $\psi(v)$.  
        This takes $n-1$ moves, and it then takes $n-2$ more to move the rest of the vertices to $\psi(G)$. 
    
        The shortest recolouring  path from $\phi$ to $\psi$ has length $2n-3 = |\psi - \phi| + \odd(\phi,\psi) - 1$, showing that the first bound of Theorem \ref{thm:recoldistbound} is sharp.   
    \end{example}

    \begin{example}
        Let $H$ be the square product of loopless path of length $3$ and a loopless path of length $n$, shown for $n = 5$ on the right of  Figure \ref{fig}. Let $G$ be the loopless path of length $n+2$, and $p$ map its endpoints to the vertices $(1,0)$ and $(1,n)$. Let $\phi$ and $\psi$ be the the $p$-extensions show in the figure. There are only two possible shortest  recolourings from $\phi$ to $\psi$, and the both have length $|\psi - \phi| -1$.  This shows that the second bound of Theorem \ref{thm:recoldistbound} is sharp.   
 
    \end{example}

   \subsection{Efficient Dismantlings and the proof of Theorem \ref{thm:recoldistbound}} 

   In this subsection we prove Theorem  \ref{thm:recoldistbound}.  Observe that for a recolouring path to achieve the trivial lower bounds of Fact \ref{fact:lowerbound},
   at every step the vertex $v$ that moves would have to move along a shortest path from $\phi(v)$ to $\psi(v)$.  

    When $H$ has a \kNU[3] polymorphism, we almost get such a path by using what we call an efficient dismantling of $H^2$ to the diagonal.  We start with these. 
    A dismantling of a reflexive graph $H$ to a retract $r(H)$ is {\em efficient} if at every step, the dismantled vertex $v$ moves closer to $r(H)$.

    \begin{lemma}\label{lem:effdispoint}
     Any bipartite connected \kNU[3] graph has an efficient dismantling to any edge.   Any  connected \kNU[3] graph has  an efficient dismantling to any looped vertex. 
    \end{lemma}
    \begin{proof}
      The statement for bipartite graphs is immediate by induction using the following claim. 
    
    \begin{claim}\label{cl:effdispoint}
       Let $H$ be a connected bipartite graph with a \kNU[3] polymorphism, and let $e = vv'$ be a fixed edge of $H$.  Let $u$ be a vertex of $H$ with maximum distance from $e$, without loss of generality assume it is closer to $v$ than to $v'$.  There is a dismantling retraction of $u$ to some vertex $u'$ that is closer to $v'$ than $u$ is to $v$.    
    \end{claim}
    \begin{proof}[Proof of claim]\claimproof
       
        Let $d$ be the distance from $u$ to $v$ and let $U_d$ be the set of neighbours of $u$. By choice of $u$ and the fact that $H$ is bipartite, we get that every vertex of $U_d$ has a walk of length $d$ to $v'$.       
        
        If we can show that the vertices in $U_d$ have a common neighbour $u'$ at distance $d-1$ from $v'$, then $u$ can dismantle to this, and we are done. 
        
        A subset $U$ of $U_d$ is `good' if its vertices share a neighbour at distance $d-1$ from $v'$.  We show that all subsets $U$ of $U_d$ are good.  Towards contradiction, let $U$ be an inclusion minimal subset of $U_d$ that is not good.  It must contain at least two vertices, so $k = |U| + 1$ is at least $3$.

        Let $T$ be the tree with $k$ branches, the first having length $d-1$ and the others having length $1$. Let $p$ map the leaf of the long branch to $v'$ and biject the other leaves to $U$.  As $H$ is a \kNU[3] graph, it has no critical trees with $3$ or more leaves by Corollary \ref{cor:treeDuals}, so in particular, $(T,p)$ is not a critical tree. 
        
        As any subset of $U$ is good, $T \setminus \{e\}$ admits a $p$-extension to $H$ for any of the short branches $e$ of $T$.
        As the elements of $U$ have the common neighbour $u$, $T \setminus \{e\}$ admits a $p$-extension to $H$ for any edge $e$ in the long branch of $T$.  As $(T,p)$ is not critical, it must therefore admit a $p$-extending homomorphism $\phi$ to $H$. The target of the degree $k$ vertices of $T$ under this homomorphism is a neighbour of the vertices of $U$, and has distance $d-1$ to $v'$. So $U$ is good, which is our contradiction.  
    \end{proof}
    
     Having completed the proof for bipartite $H$, the proof for general graphs follows by Lemma \ref{lem:disred} by observing that if a dismantling step of $K_2 \times H$ is efficient, then the symmetric shadow dismantling step of $H$ is efficient. Indeed, the distance of a vertex $u$ from $v$ in $H$ is exactly the distance of a vertex $(i,u)$ from $(0,v)(1,v)$ in $K_2 \times H$.
     \end{proof}

    \begin{theorem}\label{thm:effdis}
      For any connected \kNU[3] graph $H$ there is an efficient dismantling of $C_\Delta(H^2)$ to $\Delta(H^2)$.  Such a dismantling can be found in polynomial time. 
    \end{theorem}
     \begin{proof}
        Again we prove the statement for bipartite graphs; the general statement follows, as in Lemma \ref{lem:effdispoint}, by taking a symmetric shadow dismantling.  
         
        Let $H$ be a connected bipartite \kNU[3] graph. It is enough to show that for any retract $R = r(H)$ of $C_\Delta(H^2)$ containing $\Delta = \Delta(H^2)$ there is a vertex $u$ of $R$ that dismantles to some vertex closer to $\Delta$. 
        
        As any vertex $(u_1,u_2)$ of $R$ is in $C_\Delta(H^2)$ the vertices $u_1$ and $u_2$ are distance $2d$ apart  for some $d$, and the closest vertices to them in $\Delta$ are vertices $(v,v)$ where $v$ is the middle vertex of a path $P$ of length $2d$ from $u_1$ to $u_2$.  In particular $d(v,u_1) = d = d(v,u_2)$.   
        
        Let $(u_1,u_2)$ be a vertex of $R$ that maximises this distance $d$, and let $P$ and $v$ be as above. The subgraph $H'$ of $H$ of consisting of vertices of distance at most $d$ from $v$ is a retract of $H$ (this follows by Claim \ref{cl:effdispoint} as one can retract any vertex at distance more than $d$ from $v$ by choosing an appropriate edge containing $v$) so is \kNU[3] by Fact \ref{fact:NUvar}. $R$ is a retract of $(H')^2$, so is also \kNU[3]. 
        
        Where $u'_2$ is the neighbor of $u_2$ on the path $P$, $v_1$ is a vertex in $R$ having maximum distance from the edge $u'_2u_2$, and so by Lemma  
        \ref{lem:effdispoint} dismantles to a vertex $u'$ closer to $u'_2u_2$.
        
        It follows that $(u_1,u_2)$ dismantles to $(u',u_2)$ in $(H')^2$ and so 
        to $r((u',u_2))$ in $R$.  As $u'$ is closer to the edge $u'_2u_2$ than $u_1$ is in $H'$, $(u',u_2)$ is closer to the diagonal than $(u_1,u_2)$ is in $R$, 
        as needed. 
        
        To find an efficient dismantling, Claim \ref{cl:effdispoint} tells us we can simply order the vertices by their distance from the diagonal, and then processing the vertices farthest from the diagonal first, find a dominating neighbour closer to the diagonal and retract to it.  This can be done in time $O(n_H^3)$.

        \end{proof}
        
     We wonder if Theorem \ref{thm:effdis} might not hold for all $\NU$ graphs.  We use it now to prove Theorem \ref{thm:recoldistbound}, the two statements of which we restate as Corollaries \ref{cor:distboundodd} and \ref{cor:distbound}.

    \begin{corollary}\label{cor:distbound}
     For bipartite \kNU[3] graph $H$ and elements $\phi$ and $\psi$ of $\bCol(G,H;p)$ we have
      $d_{\Col}(\phi,\psi) \leq |\psi - \phi| - 1$.  
    \end{corollary}
    \begin{proof}
       
       As $H$ is \kNU[3], Theorem \ref{thm:effdis} gives us an  efficient dismantling $D: r_0, r_1, \dots, r_d$ of $C_\Delta(H^2)$ to $\Delta$.
       Using this dismantling $D$ in the proof of Lemma \ref{lem:Dis1} we get the path 
                \[P: a_0, a_1, \dots, a_d,b_{d-1}, \dots, b_1, b_0\] 
        from $\phi$ and $\psi$, given as \eqref{eq:abwalk} in that proof. 

       As $H$ is bipartite and there is a path from $\phi$ to $\psi$ we have that for any vertex $v$ of $G$ the distance $d_H(\phi(v),\psi(v))$ is even.  
      \begin{claim} For a vertex $v$ of $G$ if $d_H(\phi(v),\psi(v)) = 2\ell$ the vertex $v$ moves at most $2\ell$ times under $P$. \end{claim}
      \begin{proof}\claimproof
            The proof is an induction on $\ell$. The claim is obvious if $\ell = 0$ as $\Delta$ is fixed by $D$, so assume that $\ell \geq 1$.  Where $m$ is the middle vertex of a shortest path  from $\phi(v)$ to $\psi(v)$, the vertex $(\phi(v),\psi(v))$ has distance $\ell$ from $(m,m)$ in $C_\Delta(H^2)$, so has distance $\ell$ to $\Delta$.   Let $i$ be such that $a_ia_{i+1}$ is the first edge of the path that moves $v$.  This means that $r_ir_{i+1}$ moves $(\phi(v),\psi(v))$; because this is an efficient dismantling, it moves it closer to $\Delta$.  Thus $a_{i+1}(v) = \pi_1 \circ r_i (\phi(v),\psi(v))$ and $b_{i+1}(v) = \pi_2 \circ r_i (\phi(v),\psi(v))$ are distance at most $2\ell - 2$ apart.  The claim follows by induction. 
       \end{proof} 

      As the claim is for all $v$, resolving $P$ into a recolouring path gives a path of length at most  $|\psi - \phi|$.
       To get the `$-1$' in the statement of the corollary, observe that by Fact \ref{fact:dismantreorder} and Lemma \ref{lem:effdispoint} we can choose the dismantling $D$ to the diagonal so that the vertex dismantled by the last step $r_d$ is any vertex of $C_\Delta(H^2)$ that has distance $1$ from the diagonal.  Let this be $(a, b)$ where, for some vertex $v$ of $G$, $m$ is the middle vertex of a shortest path between $\phi(v)$ and $\psi(v)$, and $a$ and $b$ are the neighbours of $m$ on this path that are closer to $\phi(v)$ and $\psi(v)$ respectively. So $v$ is moved by $a_{d-1}a_d$ and $a_db_{d-1}$.  

We can reduce the number of times that $v$ is moved by using path  \eqref{eq:abwalkB} of Lemma \ref{lem:Dis1} in the above proof rather than \eqref{eq:abwalk}; so we get that the recolouring path has length at most  $|\psi - \phi| -1$. 
    \end{proof}

    For a non-bipartite \kNU[3] graph $H$ we can apply this to the bipartite resolution $B(H)$ and the symmetric shadow dismantling 
    gives a path between $\phi$ and $\psi$.  However, if $d(\phi(v),\psi(v))$ is odd, then the distance of $(\phi(v),\psi(v))$ to the diagonal in $H^2$, and so in $C_\Delta(B(H))$, is $1/2\left( d(\phi(v),\psi(v)) + 1 \right)$.   Thus the path from $\phi$ to $\psi$ moves $v$ through $d(\phi(v),\psi(v)) + 1$ steps.  This gives the following.
    
    \begin{corollary}\label{cor:distboundodd}
       For \kNU[3] graph $H$ and elements $\phi$ and $\psi$ in $\bHom_B(G,H;p)$, we have  $d_{\Col}(\phi,\psi) \leq |\psi - \phi| + \odd(\phi,\psi) - 1$.
    \end{corollary}

   \section{Shortest Path Reconfiguration}\label{sect:SPR}
   
    The shortest path reconfiguration graph $\SP(H,u,v)$ is the graph on the set of shortest $uv$-paths in $H$ where two are adjacent if they differ on a single vertex.  The shortest path reconfiguration problem $\RSP$ asks for an instance $(H,u,v,P_\phi,P_\psi)$ if there is a path between $P_\phi$ and $P_\psi$ in $\SP(H,u,v)$.

We observed in the introduction that $\SP(H,u,v) = \Col(P_d,H;p)$ where $d$ is the distance between $u$ and $v$, so $H$ is $\SP$-trivial if we get it from an $\Ext$-trivial graph by removing loops,  
thus Theorem \ref{thm:mainfull} this gives Corollary \ref{cor:SPR}: any $\NU$ graph $H$, with loops removed, is $\SP$-trivial.

   This is far from a characterisation of $\SP$-trivial graphs, as odd cycles are $\SP$-trivial while for an odd cycle of length at least $5$, no addition of loops can make it 
   into an $\NU$ graph.  A full characterisation of the $\SP$-trivial graphs  would be interesting. It would have to deal with the following example which produces many $\SP$-trivial graphs that are not $\NU$ graphs for whatever loops are added.   
   
   \begin{example}
   If $H$ is bipartite and $\SP$-trivial, and we add one edge $e$ making an odd cycle, then it remains $\SP$-trivial.  For any two vertices $u$ and $v$, all shortest paths either use $e$ or all shortest paths do not use $e$.  If they do not, then they reconfigure, if they do, then we can reconfigure the paths from $u$ to $e$ and then reconfigure the paths from $e$ to $v$. 
   \end{example}
    
    We compare Corollary \ref{cor:SPR} with related results from \cite{Bonsma13} and \cite{GKL}.
    
    \begin{itemize}
      \item In \cite{Bonsma13} it was shown that chordal graphs are $\SP$-trivial. Any chordal graph, with loops added to every vertex, is a \kNU{} graph. So we recover this result. 
    
      \item In \cite{GKL} it was shown that a hypercube is $\SP$-trivial.
      (It was shown that $\RSP$ can be solved in polynomial time, but the algorithm shows the graph is $\SP$-trivial.)  A hypercube is a \kNU[3] graph, so we recover this result.

      \item In \cite{GKL} it was shown, more generally,  that any graph with a so-called triangle property is $\SP$-trivial.  The triangle property is  equivalent to the omission of certain critical tree obstructions with three leaves.  This class of graphs is incomparable with the class of $\NU$ graphs, but contains the class of \kNU[3] graphs, so their proof also shows that \kNU[3] graphs are $\SP$-trivial. 
    \end{itemize}
    
    All results mentioned from \cite{Bonsma13} and \cite{GKL} come with algorithms, polynomial in $|V(H)|$, for finding the shortest path in $\SP(H,u,v)$ between two vertices.  Though we find the shortest path between them in $\Hom(P_d,H;p)$ in polynomial time if $H$ is $\NU$ this does not find a shortest path between them in $\Col(P_d,H;p)$.  Doing this would be interesting.

    \subsection{Diameter of the Shortest Path Reconfiguration graph}

    A {\em $d$-instance} of $\RSP$ is an instance $(H,u,v,\phi,\psi)$ for which $ d_H(u,v) = d$. It was shown in 
   \cite{Bonsma13} that  the $\SP$-reconfiguration graph $\SP(H,u,v)$ has diameter $d-1$ for chordal graphs $H$ and $d$-instances.  In \cite{GKL} the diameter was shown to be at most $\binom{d-1}{2}$ for $d$ instances of hypercubes $H$. 

  From Proposition \ref{prop:discrel1} we get the following. 
    
    \begin{corollary}\label{cor:dinst}
     For \kNU{} graphs $H$, the $\SP$-reconfiguration graph $\SP(H,u,v)$ has diameter at most $(k-2)\binom{d+1}{2}$ for any $d$-instance $(H,u,v,\phi,\psi)$ of $\RSP$.  
    \end{corollary}
    \begin{proof}
      Recall that if there is a path of length $\ell$ from $\phi$ to $\psi$ in $\bHom(P_d,H)$ then  Proposition \ref{prop:discrel1} gives a path of length $(k-2)\ell$ between them in $\bHom(P_d,H;p)$.  The proof was simple, is it easily seen to work to show that if there is a path of length $\ell$ in $\bCol(P_d,H)$, then there is one of length $(k-2)\ell$ in $\bCol(P_d,H;p) = \SP(H,u,v)$.  So it is enough to show that there is a  path of length $\binom{d+1}{2}$ between  $\phi(P_d)$ and  $\psi(P_d)$ in $\bCol(P_d,H)$. 
  
    It is a simple exercise to show that one can recolour a irreflexive path $\phi(P_d)$ on $d+1$ to vertices to its initial edge by a recolouring sequence of length  $\floor{d/2}\ceil{d/2}$.  
    From here, we can move the $\ceil{d/2}$ vertices at the second vertex of $\phi(P_d)$ to the second vertex of $\psi(P_d)$ in $\ceil{d/2}$ recolouring steps, and then recolour  back to $\psi(P_d)$ in another $\floor{d/2}\ceil{d/2}$ steps. This gives a  path between $\phi$ and $\psi$ in $\bCol(P_d,H)$ of length 
    \[ 2 \floor{d/2}\ceil{d/2} +  \ceil{d/2} = \textstyle \binom{d+1}{2}, \] 
   as needed. 
    \end{proof}
    
    In the case that $H$ is a \kNU[3] graph, we improve this using Corollary \ref{cor:distboundodd}.
    
    \begin{corollary}\label{cor:dinst2}
     For any \kNU[3] graph $H$ and any $d$-instance $(H,u,v,\phi,\psi)$ of $\RSP$ there path between $\phi$ and $\psi$ in $\SP(H,u,v)$ of  length at most $d^2/2 -1$.
    \end{corollary}
    \begin{proof}
       In a path $P_d = v_0,v_1, \dots, v_d$ of length $d$, the distance between $\phi(v_i)$ and $\psi(v_i)$  is at most $2\cdot \min(i,d-i)$.
       Computing 
        \[ |\psi - \phi| \leq 2 \floor{d/2}\ceil{d/2} \leq d^2/2\]
       the result follows by Corollary \ref{cor:distboundodd} by observing that all of these distances are even. 
    \end{proof}
    
    We note that our bound of $d^2/2$, loses slightly in the irreflexive case to the tight bound of $\binom{d-1}{2}$ given (for hypercubes) in \cite{GKL}. But it applies to a larger class of graphs.

\providecommand{\bysame}{\leavevmode\hboxto3em{\hrulefill}\thinspace}
\newcommand{\doi}[1]{\href{http://dx.doi.org/#1}{\small\nolinkurl{DOI: #1}}}
\renewcommand{\url}[1]{\href{https://arxiv.org/abs/#1}{\small\nolinkurl{arXiv: #1}}}
\bibliographystyle{abbrvnat} 

\end{document}